	\documentclass[10pt, reqno]{amsart}

	\usepackage[left=1.25in, right=1.25in, top=1.5in, bottom=1.5in]{geometry}

	\usepackage{bm}

	\usepackage{MnSymbol}
	\usepackage{latexsym}
	\usepackage{amsthm}
	\usepackage{amscd}
	\usepackage{graphicx}
	\usepackage{enumitem}
	\usepackage[dvipsnames]{xcolor}
	\usepackage{tikz-cd}

	\usepackage{mathdots}

	\usepackage{microtype}

	\usepackage[mathscr]{euscript}

	\usepackage{import}

	\graphicspath{ {images/} }

\numberwithin{equation}{section}
\numberwithin{figure}{section}

\newtheorem{theorem}{Theorem}[section]
\newtheorem*{theorem*}{Theorem}

\newtheorem*{proposition*}{Proposition}

\newtheorem{corollary}[theorem]{Corollary}

\newtheorem{lemma}[theorem]{Lemma}
\newtheorem*{lemma*}{Lemma}

\newtheorem{lemma-definition}[theorem]{Lemma-Definition}

\newtheorem{conjecture}[theorem]{Conjecture}
\newtheorem*{conjecture*}{Conjecture}

\theoremstyle{definition}
\newtheorem{definition}[theorem]{Definition}

\newtheorem{remark}[theorem]{Remark}

\makeatletter
\newcommand{\colim@}[2]{%
  \vtop{\m@th\ialign{##\cr
    \hfil$#1\operator@font colim$\hfil\cr
    \noalign{\nointerlineskip\kern1.5\ex@}#2\cr
    \noalign{\nointerlineskip\kern-\ex@}\cr}}%
}

\newcommand{\limit}{%
  \mathop{\mathpalette\varlim@{\leftarrowfill@\scriptscriptstyle}}\nmlimits@
}
\newcommand{\colimit}{%
  \mathop{\mathpalette\varlim@{\rightarrowfill@\scriptscriptstyle}}\nmlimits@
}
\makeatother

\newcommand{\C}{\mathbf C}

\renewcommand{\P}{{\mathbf P}}
\newcommand{\Q}{{\mathbf Q}}

\newcommand{\Z}{{\mathbf Z}}

\newcommand\sca{\mathscr A}
\newcommand\scb{\mathscr B}

\newcommand\sco{\mathscr O}

\newcommand{\op}{\operatorname}

\newcommand{\PGL}{\op{PGL}}

\newcommand{\ev}{\operatorname{ev}}

\DeclareMathOperator{\Ann}{Ann}

\DeclareMathOperator{\rk}{rk}

\DeclareMathOperator{\NS}{NS}

\DeclareMathOperator{\ind}{ind}

\DeclareMathOperator{\per}{per}

\DeclareMathOperator{\Br}{Br}

\DeclareMathOperator{\Coh}{Coh}

\newcommand{\varmod}{\mathbin{\mathrm{mod}}}

\DeclareMathOperator{\Hdg}{Hdg}

\newcommand{\ch}{\mathrm{ch}}

\def\sHom{\mathop{\mathscr{H}\!\mathit{o}\! \kern .4pt \mathit{m}}\nolimits}

\def\sEnd{\mathop{\mathscr{E}\!\mathit{nd}}\nolimits}
\def\sAut{\mathop{\mathscr{A}\!\mathit{ut}}\nolimits}
\def\sInn{\mathop{\mathscr{I}\! \kern .8pt \mathit{nn}}\nolimits}

\newcommand{\an}{\mathrm{an}}

\newcommand{\Ho}{\mathrm{H}}

\renewcommand{\subset}{\subseteq}

\newcommand{\nc}{\newcommand}
\nc{\p}[2]{\frac{\partial #1}{\partial #2}}
\nc{\pp}[2]{\frac{\partial^2 {#1}} {\partial {#2} ^2}}
\nc{\pmix}[3]{\frac{\partial^2 {#1}}{\partial {#2}\, \partial {#3}}}

\newcommand{\bpm}{\begin{pmatrix}}
\newcommand{\epm}{\end{pmatrix}}

\newcommand{\bbm}{\begin{bmatrix}}
\newcommand{\ebm}{\end{bmatrix}}

	\usepackage{cite}
	\usepackage{hyperref}

	\hypersetup{colorlinks = true, 
				linktocpage = true, 
				linkcolor = MidnightBlue,citecolor=MidnightBlue,filecolor=MidnightBlue,urlcolor=MidnightBlue} 

	\def\sAut{\mathop{\mathscr{A}\! \kern .8pt \mathit{ut}}\nolimits}
	\def\sIso{\mathop{\mathscr{I}\! \kern .9pt \mathit{so}}\nolimits}
	\def\sGrpd{\mathop{\mathscr{G}\! \kern 1.6pt \mathrm{rpd}}\nolimits}
	\def\sGrp{\mathop{\mathscr{G}\! \kern 1.6pt \mathrm{rp}}\nolimits}
	\def\sPr{\mathop{\mathscr{P}\! \kern 1.6pt \mathrm{r}}\nolimits}
	\def\sPic{\mathop{\mathscr{P}\! \kern 1.6pt \mathrm{ic}}\nolimits}
	\def\sM{\mathop{\mathit{s}\! \kern 1.6pt \mathscr{M}}\nolimits}
	\def\sCat{\mathop{\mathscr{C}\! \kern .9pt \mathit{at}}\nolimits}
	\def\sHH{\mathop{\mathscr{H}\! \kern .9pt \mathscr{H}}\nolimits}

	\newcommand{\ktop}{\mathrm{K}^{\mathrm{top}}}

	{}

	\newcommand{\tors}{\mathrm{tors}}

	\newcommand{\wdots}{\kern 1.8pt \cdots \kern 1.8pt}
	\newcommand{\topo}{\mathrm{top}}
	\newcommand{\tw}{\mathrm{tw}}

	\usepackage{etoolbox}
	\makeatletter
	\patchcmd{\@adjustvertspacing}
	  {\jot\baselineskip \divide\jot 4}
	  {\jot=6pt}
	  {}{}
	\makeatother

	\setlist{topsep=4pt plus 2pt minus 5pt, itemsep=8pt plus 2pt minus 5pt}

\begin{document}

	\title{The period-index problem for complex tori}
	\author{James Hotchkiss}
	\address{Department of Mathematics, University of Michigan, Ann Arbor, MI 48109 \smallskip}
	\email{htchkss@umich.edu}

	\begin{abstract}
	We solve the period-index problem for the Brauer group of a general complex torus of dimension at least three, giving an explicit formula for the index of each Brauer class. As a consequence, the complex-analytic version of the period-index conjecture is false for infinitely many Brauer classes on a general complex torus of dimension at least three.
	\end{abstract}

	\maketitle

	\section{Introduction}
	\label{sec:intro}

	Let $X$ be a connected complex manifold. An \emph{Azumaya algebra} $\sca$ over $X$ is a sheaf of $\sco_X$-algebras which is locally isomorphic to a matrix $\sco_X$-algebra $\mathrm{M}_{n}(\sco_X)$. Since the rank of an Azumaya algebra $\sca$ is a square, one defines the \emph{degree} of $\sca$ to be $\sqrt{\rk \sca}$. Two Azumaya algebras $\sca$ and $\scb$ are \emph{Morita equivalent} if there exist vector bundles $E$ and $F$, and an isomorphism of $\sco_X$-algebras
	\[
		\sca \otimes \sEnd(E) \simeq \scb \otimes \sEnd(F).
	\]
	The Brauer group $\Br(X)$ is the group of Morita-equivalence classes of Azumaya algebras with tensor product. 

	Given a Brauer class $\alpha$, the \emph{period} $\per(\alpha)$ of $\alpha$ is its order in $\Br(X)$, which is a torsion group. The \emph{index} $\ind(\alpha)$ of $\alpha$ is the greatest common divisor of the set of degrees of Azumaya algebras of class $\alpha$. It follows from a general result of Antieau and Williams that $\per(\alpha)$ divides $\ind(\alpha)$, and that they share prime factors \cite{topos}. 

	The \emph{period-index problem} is the problem of determining an integer $\epsilon$ so that $\ind(\alpha)$ divides $\per(\alpha)^{\epsilon}$. If $X$ is projective, then $\Br(X)$ is a subgroup of $\Br(\C(X))$, with $\C(X)$ the function field of $X$, and the longstanding \emph{period-index conjecture} about Brauer groups of function fields (cf. \cite{ct_bourbaki}, \cite{lieblich_twisted_sheaves}) specializes to the following statement:
	\begin{conjecture}[Global period-index conjecture]
	\label{conj:globalpic}
		Let $X$ be a smooth, connected, projective variety over $\C$. For any $\alpha \in \Br(X)$,
		\[
			\ind(\alpha) \divides \per(\alpha)^{\dim X- 1}.
		\]
	\end{conjecture}
	The conjecture is trivial for $\dim X = 0$, and holds for $\dim X = 1$ by Tsen's theorem. The case $\dim X = 2$ is due to de Jong \cite{dejong_period_index}. According to the discriminant-avoidance theorem of de Jong and Starr \cite{dejong_starr}, Conjecture~\ref{conj:globalpic} for all complex projective varieties of dimension $\leq d$ is equivalent to the period-index conjecture for complex function fields of transcendence degree $\leq d$. Notably, the analogue of Conjecture~\ref{conj:globalpic} for topological Azumaya algebras on finite CW complexes has been introduced and studied by Antieau and Williams \cite{twisted_top}, \cite{six_complex}, \cite{top_per_ind}.

	Little is known about the period-index problem for Brauer groups of non-algebraic complex manifolds beyond the case of surfaces. When $X$ is an analytic K3 surface, an analysis of the argument given in \cite{huy_schroer} shows that period and index coincide, in accordance with de Jong's theorem. Moreover, the same argument applies to the case when $X$ is a $2$-dimensional complex torus. When $X$ is a Stein manifold, the period-index problem is equivalent to the topological period-index problem by the Grauert--Oka principle, and one may obtain bounds from the work of Antieau and Williams \cite{top_per_ind} (Remark~\ref{rem:oka}).

	If $X$ is a complex torus of dimension $g$ with $\alpha \in \Br(X)$, then there is a na\"ive upper bound on $\ind(\alpha)$ given by the \emph{annihilator} $\Ann(\alpha)$, which is the least degree of a finite isogeny $f:X' \to X$ such that $f^* \alpha = 0$. When $\NS(X) = 0$, one may compute $\Ann(\alpha)$ explicitly (Lemma~\ref{lem:annihilator_calculation}), and as $\alpha$ ranges over $\Br(X)[n]$, $\Ann(\alpha)$ attains any positive integer value which divides $n^g$. Our main result is that for $X$ general, $\Ann(\alpha)$ is also a lower bound for $\ind(\alpha)$:

	\begin{theorem}
	\label{thm:main}
	    Let $X$ be a general complex torus of dimension $g \geq 3$, and let $\alpha \in \Br(X)$. Then 
	    \[
	    	\ind(\alpha) = \Ann(\alpha)
	    \]
	\end{theorem}

	More precisely, Theorem~\ref{thm:main} holds for a complex torus $X$ with $\NS(X) = \Hdg^4(X) = 0$. The proof of Theorem~\ref{thm:main} is based on an argument of Voisin \cite{voisin} showing that a general complex torus of dimension at least three does not satisfy the resolution property, along with an analysis of the Hodge-theoretic properties of $\alpha$-twisted sheaves in the manner of \cite{mypaper}. As an immediate consequence of Theorem~\ref{thm:main}, we obtain the following result:

	\begin{corollary}	
	\label{cor:failure}
	    Let $X$ be a general complex torus of dimension $g \geq 3$. For each $n > 0$, there exists a Brauer class $\alpha$ with $\per(\alpha) = n$, $\ind(\alpha) = n^g$.
	\end{corollary}

	In particular, the extension of Conjecture~\ref{conj:globalpic} to even the most familiar compact K\"ahler manifolds is false. By contrast, we prove the period-index conjecture for abelian threefolds in forthcoming work with Perry \cite{joint}. 

	One may interpret Theorem~\ref{thm:main} as the statement that the rank of any locally free $\alpha$-twisted sheaf is divisible by $\Ann(\alpha)$. In fact, we show that the rank of an arbitrary $\alpha$-twisted coherent sheaf on $X$ is divisible by $\Ann(\alpha)$ (Theorem~\ref{thm:key}), so that using a \emph{coherent} index (Definition~\ref{def:coherent_index}) does not correct for the failure of the expected period-index bounds.

	\subsection{Organization of the paper}

	In \S \ref{sec:prelim}, we review background on Brauer groups and Severi--Brauer varieties. In \S \ref{sec:annihilator}, we establish some linear algebraic properties of the annihilator of a Brauer class. In \S \ref{sec:mainresult}, we prove our main result, Theorem~\ref{thm:key}.

	\subsection{Acknowledgements} I am grateful to Stefan Schreieder, whose suggestion that the period-index conjecture might fail for non-projective compact K\"ahler manifolds led to the present paper. In addition, I thank Alex Perry for several useful discussions.

	\section{Preliminaries}
	\label{sec:prelim}

	For the definition of the Brauer group of a complex manifold and its basic properties, we refer to \cite{schroer}. If $X$ is a connected complex manifold and $\alpha \in \Br(X)$, we recall from \S \ref{sec:intro} that the \emph{period} $\per(\alpha)$ is the order of $\alpha$ in $\Br(X)$, and the \emph{index} of $\alpha$ is given by
	\[
		\ind(\alpha) = \gcd(\deg \sca : [\sca] = \alpha),
	\]
	where $\sca$ runs over the Azumaya algebras of class $\alpha$. (We recall that the \emph{degree} of an Azumaya algebra is the square root of its rank.)

	\begin{lemma}
	\label{lem:prime_factors}
	    Let $X$ be a connected complex manifold, and $\alpha \in \Br(X)$. Then $\per(\alpha)$ divides $\ind(\alpha)$, and $\per(\alpha)$ and $\ind(\alpha)$ share prime factors. 
	\end{lemma}

	\begin{proof}
	    See \cite{topos}. 
	\end{proof}

	\begin{remark}
	\label{rem:oka}
		Suppose that $X$ is a connected Stein manifold, and consider the homomorphism
		\[
			\Br(X) \to \Ho^3(X, \Z)^{\tors}.
		\]
		By the Grauert--Oka principle \cite[\S 8.2]{oka}, the set of isomorphism classes of holomorphic and continuous torsors under $\PGL_{n + 1}$ coincide, for each $n$. It follows that the map above is an isomorphism, and that the period-index problem for $X$ is equivalent to the topological period-index problem for the CW complex underlying $X$. Hence, one may obtain bounds from \cite{top_per_ind}. For instance, 
		\[
			\ind(\alpha) \divides \per(\alpha)^{\dim X - 1}
		\]
		if $\per(\alpha)$ is prime to $(\dim X - 1)!$.
	\end{remark}

	It is frequently useful to work with $\alpha$-twisted sheaves in the sense of C{\u{a}}ld{\u{a}}raru \cite{cald_thesis}. We write $\Coh^\alpha(X)$ for the abelian category of $\alpha$-twisted sheaves (where, as is customary, we often suppress the choice of a cocycle representative for $\alpha$). One has the following well-known result:

	\begin{lemma}
	\label{lem:twisted_sheaf_index}
	    Let $X$ be a complex manifold, and let $\alpha \in \Br(X)$. Then 
	    \[
	    	F \mapsto \sEnd(F)
	    \]
	    gives a bijection between the set of isomorphism classes of locally free $\alpha$-twisted sheaves, and the set of isomorphism classes of Azumaya algebras on $X$ of class $\alpha$.
	\end{lemma}

	\begin{proof}
	    We refer to \cite[Theorem 1.3.5]{cald_thesis}, which holds without change in the analytic case. 
	\end{proof}

	From Lemma~\ref{lem:twisted_sheaf_index}, the index of $\alpha$ coincides with the greatest common divisor of the ranks of locally free $\alpha$-twisted sheaves. It is also natural to consider the ranks of arbitrary $\alpha$-twisted coherent sheaves, which leads to the following variant of the index:

		\begin{definition}
	\label{def:coherent_index}
	    Let $X$ be a complex manifold, with $\alpha \in \Br(X)$. The \emph{coherent index} $\ind_{\Coh}(\alpha)$ is the greatest common divisor of the ranks of $\alpha$-twisted coherent sheaves on $X$.
	\end{definition}

	We note that $\ind_{\Coh}(\alpha)$ divides $\ind(\alpha)$. It follows from Theorem~\ref{thm:key} below that $\ind_{\Coh}(\alpha)$ and $\ind(\alpha)$ coincide when $X$ is a general complex torus, but we do not know if they coincide for Brauer classes on arbitrary complex manifolds.

	\begin{remark}
		A famous question of Grothendieck \cite{grothendieck} asks when the inclusion
		\[
			\Br(X) \subset \Ho^2(X, \sco_X^{\times})^{\tors}
		\]
		is an equality. While the answer is positive in the projective case \cite{dj_gabber}, the question remains widely open in the compact K\"ahler case. It is known, however, for compact K\"ahler surfaces (cf. \cite{huy_schroer} and \cite{schroer}), as well as complex tori (cf. \cite{projective_complex_torus}, or one may deduce it from the proof of Lemma~\ref{lem:annihilator_index_bound} below).
	\end{remark}

		\begin{lemma}
	\label{lem:annihilator_index_bound}
	    Let $f:X' \to X$ be a finite cover of complex manifolds. Given $\alpha \in \Br(X)$ such that $f^* \alpha = 0$,
	    \[
	    	\ind(\alpha) \divides \deg f.
	    \]
	\end{lemma}

	\begin{proof}
	    Let $\tilde \alpha$ be a cocycle representative for $\alpha$, and let $L$ be a $f^*\tilde \alpha$-twisted line bundle on $X'$. Then $f_* L$ is a locally free $\tilde \alpha$-twisted coherent sheaf of rank $\deg f$ on $X$, and we conclude by Lemma~\ref{lem:twisted_sheaf_index}.
	\end{proof}

	In order to apply an analytic result in the proof of Theorem~\ref{thm:key} below, it will be convenient to treat $\alpha$-twisted sheaves on $X$ as genuine sheaves on a Severi--Brauer variety. Recall that a Severi--Brauer variety over a connected complex manifold $X$ is a smooth, proper $\PGL_{n + 1}$-equivariant morphism $\P \to X$ (where the action on $X$ is trivial), which is locally on $X$ of the form $\P^n \times U \to U$, for $U \subset X$. A Severi--Brauer variety $\P \to X$ has a Brauer class $[\P] \in \Br(X)$, which obstructs the existence of a holomorphic vector bundle $E$ on $X$ such that $\P \simeq \P(E)$.

	We warn the reader that we follow Giraud's convention \cite[Example V.4.8]{giraud} for the Brauer class of a Severi--Brauer variety. If one adopts the opposite convention, then the category of weight-$k$ sheaves defined below is equivalent to the category of $\alpha^{-k}$-twisted sheaves.

	\begin{definition}
	\label{def:1_twisted}
	    Let $X$ be a complex manifold, and let $\P \to X$ be a Severi--Brauer variety. A coherent sheaf $E$ on $\P$ has \emph{weight $k$} if, for any $x \in X$, the restriction of $E$ to the fiber $\P_x$ is isomorphic to $V \otimes \sco_{\P_x}(k)$ for a $\C$-vector space $V$. 

	    Let $\Coh^k(\P/X)$ be the abelian category of weight-$k$ coherent sheaves on $\P$. For simplicity, we often write $\Coh^k(\P)$, with the morphism to $X$ implicit.
	\end{definition}

\begin{remark}[Descent]
	\label{rem:descent}
		In the setting above, let $E$ be a weight-$0$ coherent sheaf on $\P$. Then the natural morphism
		\[
			\pi^* \pi_* E \to E
		\]
		is an isomorphism. 
	\end{remark}

	\begin{lemma}
	\label{lem:weight_one_vs_twisted}
	    Let $X$ be a complex manifold, and let $\alpha \in \Br(X)$, and let $\pi:\P \to X$ be a Severi--Brauer variety of class $\alpha$. There exists a rank-preserving equivalence 
	    \[
	    	\Coh^\alpha(X) \simeq \Coh^1(\P).
	    \]
	\end{lemma}

	\begin{proof}
	    Let $\tilde \alpha$ be a cocycle representative for $\alpha$. Since $\pi^* \alpha$ is trivial, there exists an $\pi^* \tilde \alpha^{-1}$-twisted line bundle $\sco_{\P}^{\tw}(1)$ on $\P$. The equivalence is given by $E \mapsto \pi^* E \otimes \sco_{\P}^{\tw}(1)$, with inverse equivalence $F \mapsto \pi_*(F \otimes \sco_{\P}^{\tw}(-1))$.
	\end{proof}

	\begin{remark}[Topologically trivial Brauer classes]
	\label{rem:top_triv}
		Let $\P \to X$ be a Severi--Brauer variety, and suppose that the Brauer class $\alpha$ of $\P$ is \emph{topologically trivial}, i.e., lies in the kernel of the homomorphism
		\[
			\Br(X) \subset \Ho^2(X, \sco_X^{\times})^{\tors} \to \Ho^3(X, \Z)^{\tors}
		\]
		arising from the exponential sequence. Then $\P \to X$ is the projectivization of a topological complex vector bundle.

		In particular, there exists a \emph{weight$-1$ topological line bundle} $\sco_{\P}^{\topo}(1)$ on $\P$, whose restriction to the fiber $\P_x$ over $x \in X$ is isomorphic to $\sco_{\P_x}(1)$. Moreover, by the Leray--Hirsch theorem, there is a decomposition 
		\[
			\ktop_0(\P) = \ktop_0(X) \oplus \ktop_0(X) \cdot [\sco_{\P}^{\topo}(1)] \oplus \cdots \oplus \ktop_0(X) \cdot [\sco_{\P}^{\topo}(1)^{\otimes d}],
		\]
		where $d$ is the relative dimension of $\P \to X$.
	\end{remark}

	\begin{definition}
	\label{def:rational_b_field}
	    Let $X$ be a complex manifold, and let $\alpha \in \Br(X)$. A \emph{rational $B$-field} for $\alpha$ is a class $B \in \Ho^2(X, \Q)$ whose image under the homomorphism
	    \[
	    	\exp(2 \pi i \cdot -):\Ho^2(X, \Q) \to \Ho^2(X, \sco_X^{\times})^{\tors}
	    \]
	    is $\alpha$. We note that a rational $B$-field for $\alpha$ exists if and only if $\alpha$ is topologically trivial.
	\end{definition}

	\begin{remark}
	\label{rem:b_field_lb}
	Let $X$ be a compact K\"ahler manifold, and let $\pi:\P \to X$ be a Severi--Brauer variety of class $\alpha$. Suppose that $\alpha$ is topologically trivial. For any weight-$1$ topological line bundle $\sco_{\P}^{\topo}(1)$ on $\P$,  
	\[
		c_1(\sco_{\P}^{\topo}(1)) = H + \pi^*B,
	\]
	where $H \in \NS(\P)_{\Q}$ is a class whose restriction to any fiber $\P_x$ is $\sco_{\P_x}(1)$, and $B \in \Ho^2(X^{\an}, \Q)$ is a rational $B$-field for $\alpha$. We refer to \cite[Lemma 5.9]{dejong_perry}, which however uses the opposite convention for the class of a Severi--Brauer variety (hence the difference in sign), and which is stated for a smooth, proper variety $X$ but holds in the compact K\"ahler case by an identical argument.
	\end{remark}

	\section{The annihilator of a Brauer class}
	\label{sec:annihilator}

	\begin{definition}
	\label{def:annihilator}
	    Let $X$ be a complex torus.
	    \begin{enumerate}
	     	\item The \emph{annihilator} of a class $\omega \in \Ho^2(X, \Z/n)$ is the least degree of a finite isogeny $f:X' \to X$ such that $f^* \omega = 0$.
	     	\item The \emph{annihilator} of a Brauer class $\alpha \in \Br(X)$ is the least degree of a finite isogeny $f:X' \to X$ such that $f^* \alpha = 0$.
	     \end{enumerate} 
	\end{definition}

	By Lemma~\ref{lem:annihilator_index_bound}, $\ind(\alpha)$ divides $\Ann(\alpha)$.

	\begin{remark}
	\label{rem:NS_zero}
		Suppose that $X$ is a complex torus with $\NS(X) = 0$. From the Kummer sequence, the exponential map
		\[
			\Ho^2(X, \Z/n) \to \Br(X)[n]
		\]
		is an isomorphism. If $\alpha \in \Br(X)[n]$, then $\alpha$ lies in the image of a unique class $\omega \in \Ho^2(X, \Z/n)$, and $\Ann(\alpha) = \Ann(\omega)$.
	\end{remark}

	The following lemma computes $\Ann(\omega)$ in terms of linear algebra. For an element $a \in \Z/n$, we write $\Ann_{\Z}(a) \in \Z$ for the positive generator of the annihilator ideal of $a$.

	\begin{lemma}
	\label{lem:annihilator_calculation}
	    Let $X$ be a complex torus of dimension $g$, with $\omega \in \Ho^2(X, \Z/n)$. Suppose that $e_1, \dots, e_{2g}$ is a basis for $\Ho^1(X, \Z)$ such that
	    \[
	    	\omega = \sum_{i = 1}^r a_i e_{i} \wedge e_{i + g},
	    \]
	    for $0 \neq a_i \in \Z/n$ and $0 \leq r \leq g$. Then 
	    \[
	    	\Ann(\omega) = \prod_{i = 1}^r \Ann_{\Z}(a_i).
	    \]
	\end{lemma}

	\begin{proof}
	First, consider an isogeny $f:X' \to X$ of degree $\prod \Ann_{\Z}(a_i)$ such that 
    \[
    	f^* e_{i} = \Ann(a_i) \cdot e'_{i}, \quad 1 \leq i \leq r
    \]
    where $e'_1, \dots, e'_{2g}$ is a basis for $\Ho^1(X', \Z)$. Then $f^* \omega = 0$, so $\Ann(\omega) \leq \prod \Ann_{\Z}(a_i)$. 

	In the other direction, we may suppose that $n = p^e$ is a prime power. Let
	\[
	 	\eta = \sum_{i = 1}^r a_i' e_{i} \wedge e_{i + g} \in \Ho^2(X, \Z)
	 \] 
	 be a lift of $\omega$, and let $f:X' \to X$ be a finite isogeny such that $f^* \omega = 0$. Then $f^* \eta$ is divisible by $n$. In particular, 
    \begin{align*}
    	f_* f^*\eta^r/r! &= \deg f \cdot \prod_{i = 1}^r a'_i \\
    	 &\equiv 0 \mod n^r.
    \end{align*}
    Since $n$ is a prime power, it follows that $\prod \Ann_{\Z}(a_i)$ divides $\deg f$.
	\end{proof}

	If $X$ is a complex torus, then $\Ho^3(X, \Z)^{\tors} = 0$, so any Brauer class is topologically trivial and admits a rational $B$-field.

	\begin{lemma}
	\label{lem:rational_b_field_annihilator}
	    Let $X$ be a complex torus with $\NS(X) = 0$, and let $\alpha \in \Br(X)$. For any rational $B$-field $B$ for $\alpha$, let $N$ be the least positive integer such that 
	    \[
	    	N \cdot \exp(B) \in \Ho^{\ev}(X, \Q)
	    \]
	    lies in $\Ho^{\ev}(X, \Z)$. Then $N = \Ann(\alpha)$. 
	\end{lemma}

	\begin{proof}
		Let $n = \per(\alpha)$. We write
		\[
			B = \sum_{i = 1}^r \frac{m_i}{n} e_{i} \wedge e_{i + g}, \quad m_i \neq 0 \in \Z, \quad 0 \leq r \leq g.
		\]
		Then the image $\omega$ of $n \cdot B$ in $\Ho^2(X, \Z/n)$ maps to $\alpha$ under the exponential. By Remark~\ref{rem:NS_zero}, $\Ann(\alpha)$ is equal to $\Ann(\omega)$. By Lemma~\ref{lem:annihilator_calculation}, $\Ann(\omega) = \prod \Ann_{\Z}(m_i \varmod n)$.

		With $N$ as in the statement of the lemma, we observe that $N \cdot B^r/r!$ is integral. For each prime factor $p$ of $n$, with $v_p(m) = e$, it follows that
		\[
			N \cdot \prod m_i \equiv 0 \mod p^{er},
		\]
		so $\prod \Ann_{\Z}(m_i \varmod{p^e})$ divides $N$ for all such $p$, which implies that $\prod \Ann_{\Z}(m_i \varmod n)$ divides $N$. In the other direction, it is straightforward to show that 
		\[
			\prod \Ann_{\Z}(m_i \varmod n) \cdot \exp(B)
		\]
		is integral, so that $N$ divides $\prod \Ann_{\Z}(m_i \varmod n)$.
	\end{proof}

	\section{The main result}
	\label{sec:mainresult}

	In this section, we prove our main result, phrased in terms of the coherent index of Definition~\ref{def:coherent_index}.

	\begin{theorem}
	\label{thm:key}
	    Let $X$ be a complex torus such that $\NS(X) = \Hdg^4(X) = 0$, with $\alpha \in \Br(X)$. Then
	    \[
	    	\ind_{\Coh}(\alpha) = \Ann(\alpha)
	    \]
	\end{theorem}

	Theorem~\ref{thm:key} implies Theorem~\ref{thm:main}, and also implies Corollary~\ref{cor:failure} by the formula given for $\Ann(\alpha)$ through Remark~\ref{rem:NS_zero} and Lemma~\ref{lem:annihilator_calculation}. We note that if a complex torus $X$ satisfies the condition that $\NS(X) = \Hdg^4(X) = 0$, and $E$ is a locally free coherent sheaf on $X$, then $c_i(E) = 0$ for each $i > 0$ \cite{voisin}. In particular, $X$ does not have the resolution property.

	\begin{proof}
	    First, from the definition of $\ind_{\Coh}(\alpha)$ and Lemma~\ref{lem:annihilator_index_bound},
	    \[
	    	\ind_{\Coh}(\alpha) \divides \ind(\alpha) \divides \Ann(\alpha),
	    \]
	    so it remains to show that $\Ann(\alpha)$ divides $\ind_{\Coh}(\alpha)$, i.e., that the rank of any $\alpha$-twisted coherent sheaf is divisible by $\Ann(\alpha)$. 
	    
	    Let $\P \to X$ be a Severi--Brauer variety of class $\alpha$, and let $E$ be a coherent sheaf of weight $1$ on $\P$. By Lemma~\ref{lem:weight_one_vs_twisted}, it is enough to show the rank of $E$ is divisible by $\Ann(\alpha)$. We may suppose that $E$ is torsion-free, and by induction on the rank, we may suppose that $E$ contains no nonzero subsheaves of smaller rank. Then $E^{\vee \vee}$ contains no nonzero subsheaves of smaller rank as well, and is reflexive. We replace $E$ with $E^{\vee \vee}$.

	    Let $f:X' \to X$ be a finite isogeny such that $f^* \alpha = 0$, and consider a pullback diagram
	    	\[
					\begin{tikzcd}
						\P' \ar[d] \ar[r, "g"] & \P \ar[d] \\
						X' \ar[r, "f"] & X.
					\end{tikzcd}
				\]
			We note that $E$ is $\mu$-stable with respect to a K\"ahler form $\omega$ on $\P$, so by \cite[Prop. 2.3]{finite_map}, $g^* E$ is $\mu$-polystable with respect to $g^* \omega$. Since $\P' \to X'$ is a projective bundle, there is a holomorphic relative hyperplane bundle $\sco_{\P'}(1)$ so that $E' = g^* E \otimes \sco_{\P'}(-1)$ has weight $0$, and in particular descends to a coherent sheaf $E'_0$ on $X'$ (Remark~\ref{rem:descent}). 

			Our assumption on $X$ implies that $\NS(X') = \Hdg^4(X') = 0$, so that $c_1(E') = c_2(E') = 0$. Since $E'$ is reflexive and $\mu$-polystable, the theorem of Bando and Siu \cite[Corollary 3]{bando_siu} implies that $E'$ is a flat holomorphic vector bundle, and in particular, $c_i(E') = 0$ for all $i > 0$.

			Next, we compute Chern characters. Let $\sco_{\P}^{\topo}(1)$ be a $1$-twisted topological line bundle on $\P$. From the Leray--Hirsch theorem (Remark~\ref{rem:top_triv}), we may write
			\[
				\ch(E) = \ch(\sco_{\P}^{\topo}(1)) \cdot \ch(\xi),
			\]
			where $\xi \in \ktop_0(X)$. From Remark~\ref{rem:b_field_lb}, we may write $c_1(\sco_{\P}^{\topo}(1)) = H + B$, where $H$ is a rational Hodge class and $B$ is a rational $B$-field for $\alpha$. Since $H$ is Hodge, 
			\[ 
				\zeta = \exp(B) \cdot \ch(\xi) \in \Ho^{\ev}(X, \Q)
			\]
			is Hodge in each degree. 

			We claim that $f^* \zeta = \ch(E'_0)$. To prove the claim, first observe that $f^* B$ is a rational $B$-field for $f^* \alpha = 0$, hence must be integral because $\NS(X') = 0$. Therefore, $g^* h$ is an integral Hodge class, and generates $\NS(\P')$. It follows that $c_1(\sco_{\P'}(1)) = g^* h$, which implies the claim.

			From the first part of the proof, 
			\[
				\ch(E'_0) = N \cdot \mathbf{1} \in \Ho^{\ev}(X', \Z),
			\]
			for an integer $N \geq 0$, where $\mathbf{1}$ is the multiplicative unit. It follows that $\zeta = N \cdot \mathbf{1}$. 

			It remains to show that $\Ann(\alpha)$ divides $N$. But 
			\[
				\ch(\xi) = N \cdot \exp(-B)
			\]
			lies in $\Ho^{\ev}(X, \Z)$, so we conclude by Lemma~\ref{lem:rational_b_field_annihilator}.
	\end{proof}

	\bibliography{mybib}
	\bibliographystyle{amsalpha}

\end{document}